\documentclass[12pt,a4paper]{article}
\usepackage{amsmath,amsfonts,amssymb,amsthm}
\RequirePackage[dvips]{graphicx}

\begin{document}

\renewcommand*{\proofname}{\bf Proof}
\newtheorem{thm}{Theorem}
\newtheorem{utv}{Statement}
\newtheorem{cor}{Corollary}
\newtheorem{lem}{Lemma}
\theoremstyle{definition}
\newtheorem{defin}{Definition}
\theoremstyle{remark}
\newtheorem{rem}{\bf Remark}

\sloppy
\righthyphenmin=2
\exhyphenpenalty=10000
\binoppenalty=8000
\relpenalty=8000

\def\q#1.{{\bf #1.}}
\def\N{\mathbb N}
\def\Z{\mathbb Z}
\def\Q{\mathbb Q}
\def\R{\mathbb R}
\def\Pp{\mathbb P}
\def\C{{\mathbb C}}

\title{On gluing a surface of genus~$g$ from two and three polygons}
\author{A.\,V.\,Pastor
}
\date{}
\maketitle

\section{Introduction}

\subsection{Gluing together polygons into a  surface}

Consider  $k$ disks $D_1,D_2,\ldots,D_k$. Let $2n$ points be marked on the boundary circles $D_1$,\dots, $D_k$ such that at least one point is marked on each circle. 
We fix on  each circle~$D_i$ a  counterclockwise orientation. The marked points divide the circle~$D_i$ into several arcs (which do not contain marked points). One of these arcs  is marked with the number~$i$.

Thus we have $2n$ arcs on $k$ circles. We divide these arcs into pairs and glue together correspondent arcs  such that two arcs in each pair are oppositely oriented.  We obtain as a result of gluing a compact orientable surface without boundary. This surface can be disconnected. 
Points marked on the circles and arcs of these circles form a graph drawn on obtained surface.

In what follows we call  disks  $D_1,D_2,\ldots,D_k$ by  {\em polygons}. Then  marked points are  {\em vertices}  and arcs of our circles   are {\em edges} of these polygons. Let  a  disk with exactly $m$  marked points on its boundary circle be called $m$-gon.  We 
allow polygons with one and two vertices. For each polygon we fix  the {\em first} and the {\em last} vertex (in counterclockwise orientation).

Set the following notations. Let $m_1,\ldots,m_k\in\N$, $m_1+\ldots+m_k=2n$ and $D_1,\ldots,D_k$ be polygons containing $m_1,\ldots,m_k$ edges, respectively. 
We denote by~$\varepsilon_g(m_1,\ldots,m_k;k)$ the number of ways to glue these $k$ polygons  into a connected orientable surface of genus $g$
(i.e. the number of ways to divide all  $2n$ edges of these polygons into pairs  and glue together edges in each pair such that a connected orientable surface of genus $g$ will be obtained).
Let $$\varepsilon_g(n,k)=\sum_{m_1+\ldots+m_k=2n}\varepsilon_g(m_1,\ldots,m_k;k).$$
That is  $\varepsilon_g(n,k)$ is the number of ways to glue together $k$ polygons, that have (together)  $2n$ edges,  into a  connected orientable surface of genus $g$.  

For $n=0$ we set $\varepsilon_0(0,1)=1$ and $\varepsilon_g(0,k)=0$ if $g+k>1$.

Following the notations from the papers~\cite{APRW1} and~\cite{APRW2}, we set
   $$
   \mathbf{C}_g^{[k]}(z) = \sum_{n\ge0}\varepsilon_g(n,k)z^n.
   $$
For $k=1$ we will write $\varepsilon_g(n)$ and $\mathbf{C}_g(z)$ instead of $\varepsilon_g(n,1)$  and $\mathbf{C}_g^{[1]}(z)$, respectively.

\bigskip

Gluings of  a surface of genus~$g$ from one polygon were for the first time considered by J.\,Harer and D.\,Zagier~\cite{HZ}.  In this paper the numbers of such gluings   were used to  calculate Euler's characteristic of moduli space. The following recurrence formula for the numbers $\varepsilon_g(n)$ was proved in~\cite{HZ}:
\begin{equation}
  \varepsilon_g(n)=\frac{2n-1}{n+1}(2\varepsilon_g(n-1)+(n-1)(2n-3)\varepsilon_{g-1}(n-2)).
  \label{eq:eps}
\end{equation} 

Many proofs of formula~\eqref{eq:eps} are known now. In the paper~\cite{GN} a bijective proof for this formula was given.

One can find explicit formulas for the numbers~$\varepsilon_g(n)$for small~$g$ in the papers~\cite{GSh} and~\cite{Adr}. 

Generating functions $\mathbf{C}_g(z)$ and $\mathbf{C}_g^{[2]}(z)$were considered in the papers~\cite{APRW1}  and~\cite{APRW2}. It was proved in~\cite{APRW1} that for~$g>0$ 
\begin{equation}
\label{eq:C1}
 \mathbf{C}_g(z)=\frac{P_g(z)}{(1-4z)^{3g-\frac{1}{2}}},
\end{equation}
where $P_g(z)$ is a polynomial with integer  	coefficients of degree at most~$3g-1$, divisible by~$z^{2g}$. In addition, $P_g({1\over 4})\ne0$.

For $g=0$ it is well known that $\varepsilon_0(n)$ is  $n$-th Catalan number, hence,
\begin{equation}
\label{eq:catalan}
 \mathbf{C}_0(z)=\frac{1-\sqrt{1-4z}}{2z}.
\end{equation}

In the paper~\cite{APRW2} the similar formula for the generating function $\mathbf{C}_g^{[2]}(z)$ was obtained. Namely, it was proved, that for $g\ge0$
\begin{equation}
\label{eq:C2}
 \mathbf{C}_g^{[2]}(z)=\frac{P_g^{[2]}(z)}{(1-4z)^{3g+2}},
\end{equation}
where $P_g^{[2]}(z)$ is a polynomial with integer  coefficients of degree at most $3g+1$, divisible by~$z^{2g+1}$. In addition, $P_g^{[2]}({1\over 4})>0$.
Moreover,   the polynomial~$P_g^{[2]}(z)$ was expressed in terms of polynomials~$P_h(z)$. In~\cite{APRW2} explicit formulas for the numbers $\varepsilon_0(n,2)$, $\varepsilon_1(n,2)$ and $\varepsilon_2(n,2)$ were obtained as a consequence.

Another proof of the explicit formulas for~$\varepsilon_0(n,2)$ and $\varepsilon_1(n,2)$ was presented in~\cite{PR}. Also in~\cite{PR} the explicit formula for the numbers~$\varepsilon_0(n,3)$ was proved.

Similar values depending from greater number of parameters were considered in some papers. For example, in~\cite{GS} the numbers~$a_{p,q,k}^{(s)}$ were considered. This numbers can be interpreted as the number of ways to glue a  $p$-gon and a $q$-gon such that
exactly~$s$ pairs of edges from distinct polygons are glued together and a  surface  with a graph on $k$ vertices is obtained. 
A formula for the generating function $A_{p,q}^{(s)}(x)=\sum_{k\ge1}a_{p,q,k}^{(s)}x^k$ was obtained in~\cite{GS}, but this formula is tedious. No explicit or recurrence formula for the numbers $a_{p,q,k}^{(s)}$ was derived from this formula for generating functions. 

In the paper~\cite{AAPZ} the numbers $N_{g,k}^l(\pmb{b},\boldsymbol\ell,\pmb{n})$, were considered, where
$\pmb{b}=(b_1,b_2,\ldots)$, $\boldsymbol\ell=(\ell_1,\ell_2,\ldots)$ and $\pmb{n}=(n_0,n_1,\ldots)$ are sequences of nonnegative integers that satisfy some conditions. The number $N_{g,k}^l(\pmb{b},\boldsymbol\ell,\pmb{n})$ can be interpreted as the number of ways to glue together several polygons into a connected orientable surface of genus~$g$ with boundary,  where~$k$ is the number of pairs of edges which are glued together and~$l$ is the number if edges that are not glued. In addition, the following condition must hold:  for all~$i$ the number of $i$-gons in the collection of polygons we glue together is equal to~$b_i$ and the number of connected components of the boundary, which contain exactly~$i$ edges, is equal to~$n_i$. Moreover, the number of vertices that lie on the boundary must be equal to~$n_0$: they are considered as components with 0 edges.  The parameter~$\ell_i$ is defined in more complicated way: it is equal to the number of connected components of the boundary and vertices that do not lie on the boundary, for which the following value is equal to~$i$. For a vertex that does not lie on the boundary this value is the sum of its degree and the number of marked vertices of the polygons, which gluing form this vertex (one vertex is marked in each polygon). For a connected component of the boundary we contract this component to a vertex (all edges of the boundary are deleted after this operation) and count for the  obtained vertex the same value as above. Some formulas for these numbers were proved in the paper~\cite{AAPZ}.

In this paper we give an elementary proof for the formula~\eqref{eq:C2} and obtain a similar formula for~$\mathbf{C}_g^{[3]}(z)$. As a consequence we obtain an explicit formula for~$\varepsilon_1(n,3)$.

\subsection{Equivalent reformulations}

The problem of counting gluings of a surface of genus~$g$ from $k$ polygons has some equivalent reformulations. We consider such reformulations that will be useful in our paper. One can see other reformulations (in particular, the counting of chord diagrams) in the papers~\cite{APRW1,APRW2,AAPZ}.

\subsubsection{Counting   marked maps}

\begin{defin}
We call by a {\em map} an ordered pair $(X,G)$, where~$G$ is a finite undirected graph (maybe, with loops and multiple edges), embedded into a compact orientable surface~$X$ without a boundary such that all connected components of the set $X \setminus G$ (named {\em faces} of this map) are homeomorphic to disks. 

Two maps $(X,G)$ and $(X',G')$ are {\em isomorphic}, if there is a homeomorphism $f: X \to X'$ which preserves orientation, such that $f(G) = G'$.

The {\em genus} of the map~$(X,G)$ is the genus of the surface~$X$.
\end{defin}

Many facts on maps (in particular, about connection between maps and  permutations) can be found in~\cite{CM}. We recall the notions that are necessary for our paper.

It is easy to see that the procedure of gluing a surface from $k$ polygons described above gives a map with $n$ edges and  $k$ faces. Faces of the obtained map correspond to polygons which were glued. We enumerate the faces with integers from 1 to $k$ as the correspondent polygons. Any edge of a map corresponds to two edges of polygons. Assign to each edge of our map a pair of oppositely oriented arcs.
We say that an arc~$e$ {\em belongs} to a face~$F$, if $e$ lies on the boundary of~$F$ and is oriented in the  counterclockwise orientation of this boundary. Then we orient  edges of polygons counterclockwise and obtain an essential bijection between the set of arcs of the map and the set of all edges of polygons. The arcs of the map which correspond to the marked edges of polygons we mark with the same numbers.

Thus we obtain a map with~$n$ edges and $k$ faces, these faces are enumerated with integers from~1 to~$k$ and for each face some arc that belongs to it is marked with the  number of the face. We say that such map is {\em marked}.
Two marked maps $(X,G)$ and $(X',G')$ are {\em isomorphic}, if there is a homeomorphism $f: X \to X'$, such that $f(G) = G'$, which preserves orientation and marks on the arcs.

Note that the procedure of gluing described above defines the map uniquely up to isomorphism.
Any map can be uniquely up to isomorphism glued from some collection of polygons. 
Thus   $\varepsilon_g(n,k)$ can be interpreted as the number of marked maps with $n$ edges and $k$ faces on a  connected surface of genus~$g$. Under this interpretation the number $\varepsilon_0(0,1)$ corresponds to a sphere with one marked point on it.

In what follows the number of  marked maps  is counted up to isomorphism, i.e. we count the number of classes of isomorphic marked maps. Similarly, speaking about marked maps of some type we mean classes of isomorphic maps of this type.

\begin{rem}
\label{rem:zero}
By Euler's formula, the number of vertices of a map is equal to  $n-k+2-2g$. It must be positive. Hence, for $n<k+2g-1$  we have $\varepsilon_g(n,k)=0$.
\end{rem}

\subsubsection{Maps end permutations}

Consider a map $(X,G)$, where the  graph $G$ has no isolated vertices. Assign to each edge of~$G$ a pair of oppositely directed arcs and denote the set of obtained arcs by~$A$. 

We set a cyclic order of outgoing arcs for every vertex of the graph~$G$.
As a result, a permutation $\sigma$ on the set $A$ of all arcs of the graph~$G$ is obtained: $\sigma(e)$ is the next (in the cyclic order) arc outgoing from the beginning of an arc $e$. 

We also define  permutations $\iota,\tau\in S(A)$ as follows. Let $\iota(e)$ be the arc opposite to~$e$, i.e. correspondent to the same edge of $G$ but oppositely oriented. Let $\tau=\sigma\iota$. Since $\iota^2=1$, we also have  $\sigma=\tau\iota$.

\begin{rem}
\label{rem:face}
Let $\sigma,\iota,\tau$ be the permutations on the set of arcs~$A$ of the map $(X,G)$ defined above,  $F$ be a face of the map $(X,G)$ and an arc $e\in A$ belong to~$F$. Then the arc~$\tau(e)$ also belongs to~$F$ and is the next arc for~$e$ in the counterclockwise walk around the boundary of~$F$ (see figure~\ref{ris:tau}).

Hence we call the arc $\tau(e)$ by the {\em next} arc  for~$e$, and the arc $\tau^{-1}(e)$ by the  {\em previous} arc for~$e$. 
Note that cycles of the permutation~$\sigma$ correspond to vertices of the map $(X,G)$,  cycles of the permutation~$\iota$ correspond to edges of the map $(X,G)$ and cycles of the permutation~$\tau$ correspond to faces of the map $(X,G)$. 
\end{rem}

\noindent
\begin{figure}[!hb]
\centerline{\includegraphics{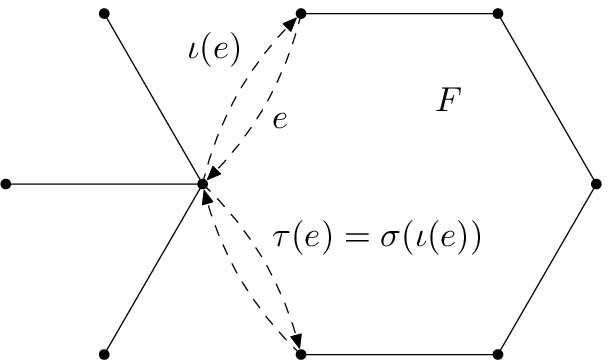}}
\caption{disposition of the arcs~$e$ and~$\tau(e)$. Edges are drawn  by solid lines and arcs --- by dotted lines.}
\label{ris:tau}
\end{figure}

\begin{rem}
\label{rem:map_permut}
Let $(A,\sigma,\iota)$ be ordered   triple,  where~$A$ is a set with even number of elements,  $\sigma,\iota$ are permutations on this set such that all cycles of~$\iota$ have length~2. It is well known (see, for example,~\cite{CM}), that for any such triple there exists a unique up to isomorphism map without isolated vertices, for which~$A$  is the set of arcs and~$\sigma$, $\iota$ are permutations defined above. 

Note that since  $\sigma=\tau\iota$, the triple $(A,\iota,\tau)$ also defines unique up to isomorphism map without isolated vertices (cycles of~$\iota$ correspond to edges and   cycles of~$\tau$ correspond to faces of this map).
\end{rem}

Now we can write down the equivalent formulation of the problem on counting marked maps.  Let's enumerate arcs of a marked map with integers from 1 to $2n$ in the following essential order:
at first we enumerate arcs of the face 1, beginning at the arc marked with 1 in the counterclockwise cyclic order, then we similarly enumerate arcs of the face~2, and so on.
Thus the permutations $\iota$ and $\tau$ defined above act on the set~$\{1,2,\ldots,2n\}$. 
Moreover,  $\iota$ consists of  $n$ independent transpositions and $\tau$  consists of  $k$ independent cycles, such that elements of each cycle are successive integers.  By remark~\ref{rem:map_permut}  such  pair of permutations defines up to isomorphism a map which edges are enumerated with integers from~1 to~$2n$ in the order defined above.  Hence the problem on counting marked maps is equivalent to the problem on counting pairs of permutations of type described above.

Cycles of the permutation~$\sigma=\tau\iota$ correspond to vertices of the map that we construct, hence, their number is $n-k+2-2g$. 
Thus $\varepsilon_g(n,k)$ is equal to the number of pairs of permutations $(\iota,\tau) \in S_{2n}\times S_{2n}$ such that
$\iota$ consists of $n$ cycles of length~2, the  group  $\langle\iota,\tau\rangle$ generated by the permutations $\iota$ and $\tau$   is transitive and $\tau\iota$ is a product of  $n-k+2-2g$ independent cycles.

\section{The operation of deleting an edge}

Our main instrument is the {\em operation of deleting an edge} defined in~\cite{PR}.  Recall the definition and basic properties of this operation.

Let $\mathcal{M}=(X,G)$ be a connected marked map of genus $g$ with $v$ vertices, $n+1$ edges and $k$ faces, where $n>0$.
Denote by $A$ the set of arcs of this map, by  $e_i$ the arc marked with~$i$, by~$f_i$ the arc opposite to~$e_i$ and by~$\tilde e_i$ the edge of the graph~$G$ correspondent to these two arcs. The permutations $\sigma$, $\iota$ and~$\tau$ defined above act on the set~$A$.

Let's delete the edge  $\tilde e_1$ from the graph $G$. If there is an isolated vertex in the obtained graph we also delete this vertex.
Denote the obtained graph by~$G'$ and the set of its arcs by~$A'$ (i.e. $A'=A\setminus\{e_1,f_1\}$). 
Define on the set~$A'$   permutations $\iota'=\iota|_{A'}$ and $\sigma'(x)=\sigma^{s(x)}(x)$, where
$s(x)=\min\{l\in\N\mid\sigma^l(x)\in A'\}$. Then the permutation~$\iota'$ acts on the elements of~$A'$ as well as~$\iota$ and the permutation $\sigma'$ acts similarly to~$\sigma$, but  omits the deleted arcs $e_1$ and~$f_1$. Note that the number of omitted arcs $s(x)-1$ can be equal to 0, 1 or 2 (see figure~\ref{ris:sigma'}).

\noindent
\begin{figure}[!hb]
\centerline{\includegraphics{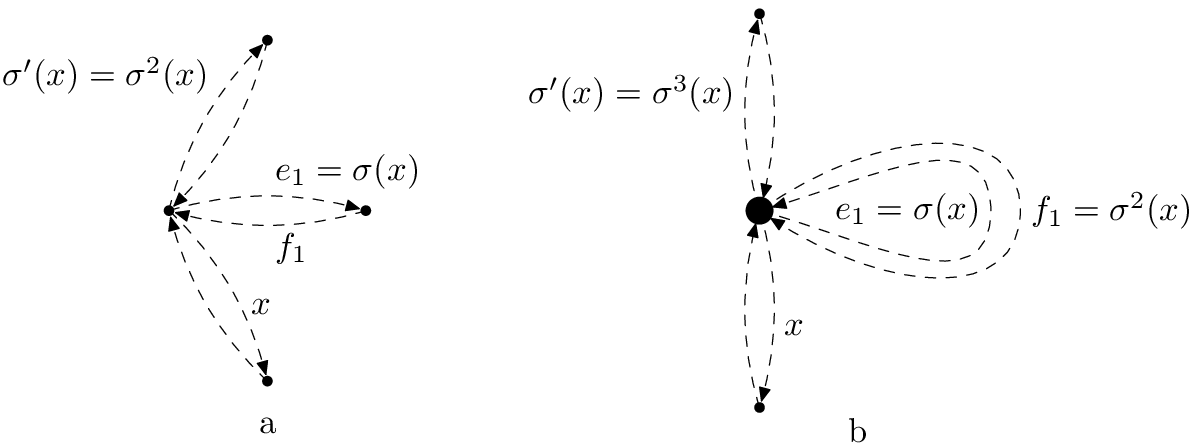}}
\caption{the case $s(x)=2$ is shown on figure~\ref{ris:sigma'}a and the case $s(x)=3$ is shown on figure~\ref{ris:sigma'}b.}
\label{ris:sigma'}
\end{figure} 

As it was noted in remark~\ref{rem:map_permut}, the ordered triple~$(A',\sigma',\iota')$ uniquely up to isomorphism defines a map~$\mathcal{M}'$, which has no isolated vertex such that~$A'$ is its set of arcs, cycles of permutation~$\iota'$ define pairs of arcs correspondent to edges and cycles of permutation~$\sigma'$ correspond to vertices of this map. Since cycles of~$\iota'$ also correspond to edges of the graph $G'$ and cycles of~$\sigma'$ correspond to vertices of~$G'$, we obtain that  $\mathcal{M}'=(X',G')$, where $X'$ is some surface. Note, that the surface~$X'$  can be disconnected.  

Consider the permutation $\tau'=\sigma'\iota'$ on the set~$A'$. As it was noted in the remark~\ref{rem:face}, its cycles correspond to faces of the map $\mathcal{M}'=(X',G')$.

\begin{lem}
\label{tau'}
Let $x\in A'$. Then
  $$
  \tau'(x)=
    \left\{
      \begin{array}{lcl}
        \tau(x)&,&\tau(x)\in A'\\
        \tau(f_1)&,&((\tau(x)=e_1\ \&\ \tau(f_1)\ne f_1)\vee(\tau(x)=f_1\ \&\ \tau(e_1)= e_1))\\
        \tau(e_1)&,&((\tau(x)=e_1\ \&\ \tau(f_1)= f_1)\vee(\tau(x)=f_1\ \&\ \tau(e_1)\ne e_1)).
      \end{array}
    \right.\,
  $$
\end{lem}

\begin{proof}
The case $\tau(x)\in A'$ is obvious. Consider the case $\tau(x)=e_1$ (the case $\tau(x)=f_1$ is similar). Let $y=\iota(x)$. Then $e_1=\sigma(y)$. Assume that $\tau(f_1)\ne f_1$. Then~$\sigma^2(y)=\sigma(e_1)=\sigma(\iota(f_1))=\tau(f_1)\in A'$ (see figure~\ref{ris:tau'}a), whence it follows that
  $$\tau'(x)=\sigma'(\iota'(x))=\sigma'(y)=\sigma^2(y)=\tau(f_1).$$
Now let $\tau(f_1)= f_1$. Then $\sigma^2(y)=\tau(f_1)=f_1\notin A'$ (see figure~\ref{ris:tau'}b), whence it follows that
  $$\tau'(x)=\sigma'(y)=\sigma^3(y)=\sigma(f_1)=\sigma(\iota(e_1))=\tau(e_1).$$
\end{proof}

\noindent
\begin{figure}[!hb]
\centerline{\includegraphics{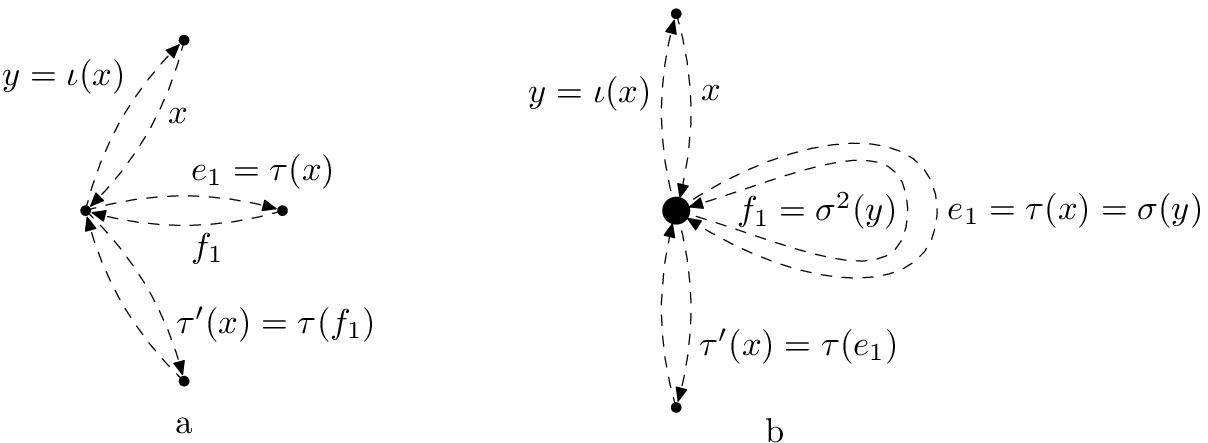}}
\caption{case $\tau(f_1)\ne f_1$ is shown on figure~\ref{ris:tau'}a, case $\tau(f_1)= f_1$ is shown on figure~\ref{ris:tau'}b.}
\label{ris:tau'}
\end{figure}

\begin{rem}
The equality $\tau(a)=a$ means that the arc~$a$ belongs to a face bounded by one edge and the edge~$\tilde{a}$ correspondent to the arc~$a$ is a loop. Note, that the equalities~$\tau(e_1)= e_1$ and $\tau(f_1)= f_1$ cannot hold together since the map~$(X,G)$ is connected and~$n>0$.
\end{rem}

Now consider cycles of the permutation~$\tau'$. Clearly, a cycle of~$\tau$  which doesn't contain~$e_1$ or~$f_1$, is a cycle of~$\tau'$. For a cycle of~$\tau$ which contains~$e_1$ or~$f_1$ the two following cases are possible: 
\begin{enumerate}
 \item $e_1$ and $f_1$ belong to distinct cycles $(e_1,a_1,\ldots,a_p)$ and $(f_1,b_1,\ldots,b_q)$;
 \item $e_1$ and $f_1$ belong to one cycle $(e_1,a_1,\ldots,a_p,f_1,b_1,\ldots,b_q)$.
\end{enumerate}
In  the first case the arcs~$e_1$ and~$f_1$ belong to different faces of the map~$(X,G)$ and in the second case~$e_1$ and~$f_1$ belong to one face of this map. In both cases one of the numbers~$p$ and~$q$ can be equal to zero: in the first case it means that one of the arcs~$e_1$ and~$f_1$ is the only arc of its face and in the second case it means that these arcs are consecutive.
Note also that in the first case the graph~$G'$ is connected. Consider these cases in details.  In each case  we count the number of faces and the genus of the map~$(X',G')$. We also put the marks on the edges  to obtain a marked map or an ordered pair of marked maps.
\medskip

\q1. \textit{$e_1$ and  $f_1$ belong to different cycles~$(e_1,a_1,\ldots,a_p)$ and $(f_1,b_1,\ldots,b_q)$.} 

It follows from lemma~\ref{tau'} that in the permutation~$\tau'$ these two cycles will be joined into one cycle 
$(a_1,\ldots,a_p,b_1,\ldots,b_q)$. Hence, in this case after deleting the edge~$\tilde{e}_1$ from the map~$(X,G)$ two faces that bound by this edge will be glued together into one face. Thus, we obtain a map with~$v$ vertices, $n$ edges and~${k-1}$ faces. By   Euler's formula its genus is equal to~$g$. To obtain a marked map we mark the new face (obtained after deleting~$\tilde e_1$) with~1, enumerate other faces with numbers from~2 to~$k-1$ in the order of increasing of their previous numbers and correct the marks on  edges. 

It remains to mark with~1 some arc of the map~$(X',G')$.  
Let the arc~$f_1$ belong to face~$j$ of the initial map. If~$e_j\ne f_1$, then we mark with~1 the arc~$e_j$. If~$e_j=f_1$ and $\tau(e_1)\ne e_1$, then we mark with~1 the arc~$\tau(e_1)$. Finally, if~$e_j=f_1$ and~$\tau(e_1)=e_1$, then we mark with~1 the arc~$\tau(f_1)$ (see figure~\ref{ris:q1}).  We obtain as a result a connected marked map of genus~$g$ with~$n$ edges and~$k-1$ faces. 

\noindent
\begin{figure}[!hb]
\centerline{\includegraphics[width=.96\columnwidth, keepaspectratio]{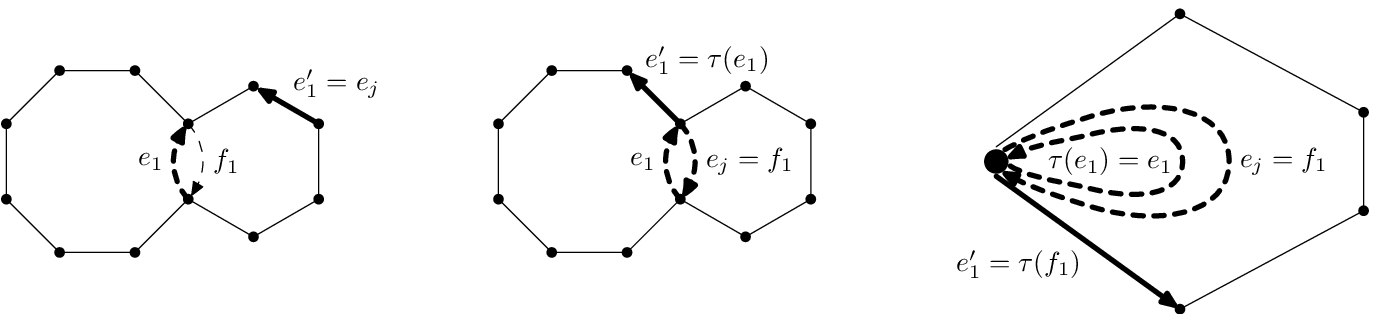}}
\caption{The operation of deleting an edge, case \textbf{1}. In each subcase the arc marked with~1 is called~$e_1'$.}
\label{ris:q1}
\end{figure} 

\q2. \textit{$e_1$ and $f_1$ belong to the cycle $(e_1,a_1,\ldots,a_p,f_1,b_1,\ldots,b_q)$.} 

Then by lemma~\ref{tau'} the permutation~$\tau'$ has the cycles $(a_1,\ldots,a_p)$ and~$(b_1,\ldots,b_q)$ (if $pq>0$  there is only one of these cycles). Hence, the map~$(X',G')$ contains $k$ faces, if~$pq=0$ and $k+1$ faces, if $pq>0$. We consider three subcases:
\begin{enumerate}
 \item[2.1.] $pq=0$ (i.e. the arcs~$e_1$ and~$f_1$ are consecutive);
 \item[2.2.] $pq>0$ and the graph~$G'$ is connected;
 \item[2.3.] $pq>0$ and the graph  $G'$ is disconnected.
\end{enumerate}
Consider these subcases in details.
\smallskip

\q{2.1}. \textit{$pq=0$.}
In this case the arcs~$e_1$ and~$f_1$ are consecutive, i.e. one of the following two equalities holds: $e_1=\tau(f_1)$ or $f_1=\tau(e_1)$. 
Then one end of the edge~$\tilde e_1$ has degree~1. After deleting the edge~$\tilde e_1$ this end becomes an isolated vertex, hence, it is also  deleted. Thus, $(X',G')$ is a connected map with~${v-1}$  vertices, $n$ edges and $k$ faces. By Euler's formula it has genus~$g$. We preserve numeration and marks on all arcs except the deleted arc~$e_1$ and mark with~1  that of the arcs~$\tau(e_1)$ and $\tau(f_1)$, which belongs to~$A'$ (see figure~\ref{ris:q21}). We obtain as a result a connected marked map of genus~$g$ with $v-1$ vertices, $n$ edges and~$k$ faces.

\noindent
\begin{figure}[!hb]
\centerline{\includegraphics{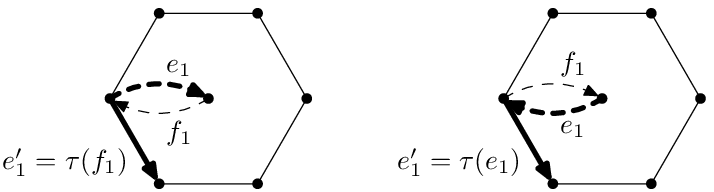}}
\caption{The operation of deleting an edge, case \textbf{2.1}. Different dispositions of the arcs~$e_1$ and~$f_1$.}
\label{ris:q21}
\end{figure}

\q{2.2}. \textit{$pq>0$ and the graph $G'$ is connected.} In this case the map $(X',G')$ is connected and contains $v$ vertices, $n$ edges and~${k+1}$ faces. By Euler's formula it has genus~${g-1}$. We enumerate with~1 the face correspondent to the cycle $(a_1,\ldots,a_p)$ of the permutation~$\tau'$ and with~${k+1}$ the face correspondent to the cycle~$(b_1,\ldots,b_q)$. We mark in this faces the arcs~$a_1$ and~$b_1$, respectively (see figure~\ref{ris:q223}a). All other marks we preserve. As a result, a connected marked map of genus~${g-1}$ with $n$ edges and~${k+1}$ faces is   obtained.

\noindent
\begin{figure}[!hb]
\centerline{\includegraphics[width=.96\columnwidth, keepaspectratio]{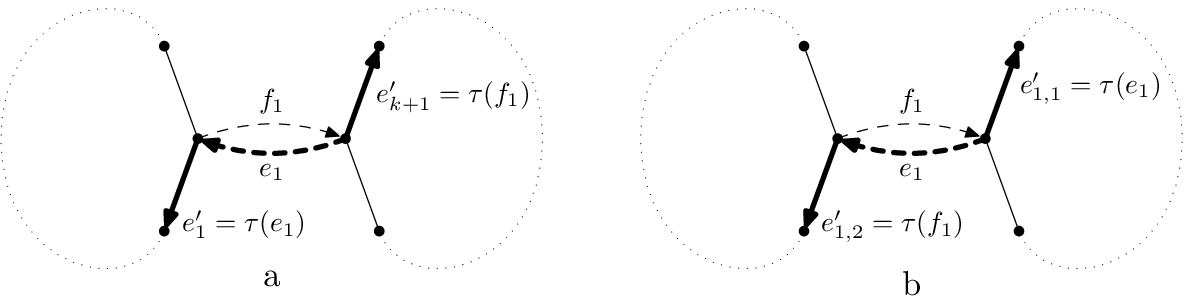}}
\caption{The operation of deleting an edge, cases \textbf{2.2} and \textbf{2.3}. In case \textbf{2.2} 
(figure~\ref{ris:q223}a) arcs, which will be marked  with~$1$ and~$k+1$, are denoted by~$e_1'$ and~$e_{k+1}'$, respectively. In case \textbf{2.3} (figure~\ref{ris:q223}b)  arcs, which will be marked with 1 in the maps $(X_1',G_1')$ and $(X_2',G_2')$, are denoted by $e_{1,1}'$ and $e_{1,2}'$, respectively.}
\label{ris:q223}
\end{figure}

\q{2.3}. \textit{$pq>0$ and the graph $G'$ is disconnected.} In this case the map $(X',G')$ consists of two connected components. Clearly,  ends of the edge~$\tilde e_1$ lie in distinct connected components. Therefore, the arcs~$a_1=\tau(e_1)$ and~$b_1=\tau(f_1)$ also lie  in distinct connected components.  Denote by~$(X_1',G_1')$ the component that contains~$a_1$ and by~$(X_2',G_2')$ the component that contains~$b_1$. We mark with~1 the arcs~$a_1$ and~$b_1$ and enumerate with~1 the faces which contain these arcs (see figure~\ref{ris:q223}b). Then we enumerate all other faces of each of components in the increasing order of their numbers in the initial map  and mark correspondent arcs with the new numbers of faces. We obtain as a result an ordered pair  of marked maps which contain together~$v$ vertices, $n$ edges and $k+1$ faces, each of maps has at least one edge. Let the first map has genus~$g_1$ and the second has genus~$g_2$.  By Euler's formula we have that $g_1+g_2=g$.

\bigskip
In~\cite{PR} for each of the cases  we have applied the operation of deleting an edge to all connected marked  maps of genus~$g$ with $n+1$ edges and $k$ faces and counted   the number of ways each resulting  map was obtained. The following lemma joins these results.

\begin{lem} {\rm \cite[lemmas 1-4]{PR}}
\label{lem:quantity}
Let's apply the  operation of deleting an edge to all connected marked  maps of genus~$g$ with $n+1$ edges and $k$ faces.  Then the following statements hold. 

$1)$ Each connected marked  map of genus~$g$ with $n$ edges and $k-1$ faces,  in which the face~$1$ has~$m$, is obtained $\frac{(m+1)(m+2)(k-1)}{2}$ times.

$2)$ Each connected marked  map  of genus~$g$ with $n$ edges and $k$ faces is obtained twice.

$3)$  Each connected marked  map  of genus~$g-1$ with $n$ edges and $k+1$ faces is obtained once.

$4)$  Each ordered pair of connected marked  maps with positive number of edges, which have together $n$ edges and $k+1$ faces and   have sum of the genera~$g$  is obtained  $C_{k-1}^{k_1-1}=C_{k-1}^{k_2-1}$ times,  where $k_1$ and~$k_2$ are the numbers of faces in the first and in the second maps, respectively.
\end{lem}

\begin{rem}
\label{rem:trivmap}
The cases \textbf{2.1} and \textbf{2.3} are quite similar. In fact, instead of a connected marked map of genus~$g$ with~$n$ edges and~$k$ faces we may consider in case~\textbf{2.1} an ordered pair of this map and a map with one vertex and zero edges i.e a sphere with one marked vertex. Such a pair appears if   instead of deleting an isolated vertex obtained after deleting the edge~$\tilde e_1$ we consider it as a trivial map.
This trivial map is the first in our ordered pair if the arc~$e_1$ is  oriented to its only vertex and the second if~$e_1$ is oriented from this vertex. In this case each of these two ordered pairs is obtained once. That corresponds to the formula of item~4 of lemma~\ref{lem:quantity}. Thus we can get rid of case~\textbf{2.1} and delete from case \textbf{2.3} the condition that each map of the pair contains at least one edge. 
\end{rem}

\begin{thm}
\label{req}
For any integers $n>0$, $k>0$ and $g\ge0$ the following equality holds: 
\begin{multline}
\label{req:all}
 \varepsilon_g(n+1,k)=\\
 =\sum_{m_1+\ldots+m_{k-1}=2n}\tfrac{(m_1+1)(m_1+2)(k-1)}{2}\varepsilon_g(m_1,\ldots,m_{k-1};k-1)+\\
 +\sum_{\ell=1}^k\sum_{h=0}^g\sum_{i=0}^nC_{k-1}^{\ell-1}\varepsilon_{h}(i,\ell)\varepsilon_{g-h}(n-i,k-\ell+1)+\\
 +\varepsilon_{g-1}(n,k+1).
\end{multline}
In the cases where~$k-1=0$ or $g-1=-1$, the correspondent summands are equal to zero.
\end{thm}

\begin{proof}
Apply the operation of deletion an edge to all connected marked  maps of genus~$g$ with $n+1$ edges and $k$ faces. 
In the cases where one map of genus~$g$ with $n$ edges and~$k$ faces is obtained we substitute it by an ordered pair of connected marked maps as it was written in remark~\ref{rem:trivmap}. Counting the number of ways each object (a map or an ordered pair of maps) is obtained as a result of this operation with the help of lemma~\ref{lem:quantity} and remark~\ref{rem:trivmap}, we get the desired equality.
\end{proof}

\begin{cor}
For $n>0$ the following equalities hold:
\begin{equation}
\label{req:eps2}
 \varepsilon_g(n,2)=\varepsilon_{g+1}(n+1)-\sum_{h=0}^{g+1}\sum_{i=0}^{n}\varepsilon_h(i)\varepsilon_{g+1-h}(n-i);
\end{equation}

\begin{multline}
\label{req:eps3}
 \varepsilon_g(n,3)=\varepsilon_{g+1}(n+1,2)-2\sum_{h=0}^{g+1}\sum_{i=0}^{n}\varepsilon_h(i)\varepsilon_{g+1-h}(n-i,2)-\\-(n+1)(2n+1)\varepsilon_{g+1}(n).
\end{multline}
\end{cor}

\begin{proof}
Formulas~\eqref{req:eps2} and~\eqref{req:eps3} follow immediately from theorem~\ref{req}: substitute~$g$ by $g+1$ in formula~\eqref{req:all} and put $k=1$ and $k=2$, respectively).
\end{proof}

\section{Gluings from two polygons}

In this section with the help of the operation of deleting an edge we give an elementary proof of the formula~\eqref{eq:C2}.

\begin{thm}[{see~\cite[theorem 4.2]{APRW2}}]
\label{th:C2}
For all $g\ge0$
\begin{equation*}
 \mathbf{C}_g^{[2]}(z)=\frac{P_g^{[2]}(z)}{(1-4z)^{3g+2}},
\end{equation*}
where $P_g^{[2]}(z)$ is a polynomial with integer coefficients, which satisfies the equation
\begin{equation}
 P_g^{[2]}(z)=z^{-1}P_{g+1}(z)-\sum_{h=1}^g P_h(z)P_{g+1-h}(z).
\end{equation}
\end{thm}

\begin{proof}
Applying formulas~\eqref{eq:C1}, \eqref{eq:catalan}, \eqref{req:eps2} and taking into account that $\varepsilon_{g+1}(0)=0$ 
by remark~\ref{rem:zero}, we obtain
\begin{align*}
 \mathbf{C}_g^{[2]}(z)&=\sum_{n\ge0}\varepsilon_g(n,2)z^n=\\
&=\sum_{n\ge0}\biggl(\varepsilon_{g+1}(n+1)z^n-\sum_{h=0}^{g+1}\sum_{i=0}^{n}\varepsilon_h(i)z^i\cdot\varepsilon_{g+1-h}(n-i)z^{n-i}\biggr)=\\
&=z^{-1}\mathbf{C}_{g+1}(z)-\sum_{h=0}^{g+1}\mathbf{C}_{h}(z)\mathbf{C}_{g+1-h}(z)=\\
&=z^{-1}\mathbf{C}_{g+1}(z)-\sum_{h=1}^{g}\mathbf{C}_{h}(z)\mathbf{C}_{g+1-h}(z)-2\mathbf{C}_{0}(z)\mathbf{C}_{g+1}(z)=\\
&=\frac{P_{g+1}(z)}{z(1-4z)^{3g+\frac{5}{2}}}-\sum_{h=1}^{g}\frac{P_h(z)P_{g+1-h}(z)}{(1-4z)^{3g+2}}-2\frac{(1-\sqrt{1-4z})P_{g+1}(z)}{2z(1-4z)^{3g+\frac{5}{2}}}=\\
&=\biggl(z^{-1}P_{g+1}(z)-\sum_{h=1}^g P_h(z)P_{g+1-h}(z)\biggr)\frac{1}{(1-4z)^{3g+2}}.
\end{align*}

It is proved in~\cite{APRW1} that for all  $h\ge1$ the polynomial $P_h(z)$ has integer coefficients and is divisible by~$z^{2h}$. Hence,~$P_g^{[2]}(z)=z^{-1}P_{g+1}(z)-\sum_{h=1}^g P_h(z)P_{g+1-h}(z)$ is also a polynomial with integer coefficients.
\end{proof}

\begin{rem}
\label{rem:p1p2}
The polynomials~$P_g(z)$ were calculated  for $g=1,\ldots,5$ in~\cite{APRW1}. We write down these polynomials.
\begin{align*}
P_1(z) &= z^2,\\
P_2(z) &= 21z^4\, \left( z+1 \right)\\
P_3(z) &=  11z^6\, \left( 158\,{z}^{2}+558\,z+135 \right),\\
P_4(z) &=143z^8\left( 2339\,{z}^{3}+18378\,{z}^{2}+13689\,z+1575 \right),\\
P_5(z) &=  88179z^{10}\, \left( 1354\,{z}^{4}+18908\,{z}^{3}+28764\,{z}^{2}+9660\,z+675 \right).
\end{align*}
The polynomials $P_g^{[2]}(z)$ were calculated  for $g=0,1,\ldots,5$ in the paper~\cite{APRW2}. We also write down them.
\begin{align*}
P_0^{[2]}(z)&=z,\\
P_1^{[2]}(z)&=z^3\, (20z+21),\\
P_2^{[2]}(z)&=z^5\, \left(1696z^2+6096z+1485 \right),\\
P_3^{[2]}(z)&=z^7\, \left(330560z^3+2614896z^2+1954116z+225225 \right),\\
P_4^{[2]}(z)&=z^9\, \left(118652416z^4 +1661701632z^3+\right.\\
&\qquad\qquad\qquad\left.+2532145536z^2+851296320z+59520825\right),\\
P_5^{[2]}(z)&=z^{11}\, \left(68602726400z^5+
1495077259776z^4+3850801696512z^3+\right.\\
&\qquad\qquad\qquad\left.+2561320295136z^2 +505213089300z+24325703325\right).
\end{align*}
\end{rem}

\section{Gluings from three polygons}

In this section we prove a formula similar to~\eqref{eq:C2}  for the generating function~$\mathbf{C}_g^{[3]}(z)$.

\begin{thm}
\label{th:C3}
For all $g\ge0$
\begin{equation}
\label{C3}
 \mathbf{C}_g^{[3]}(z)=\frac{P_g^{[3]}(z)}{(1-4z)^{3g+4,5}},
\end{equation}
where $P_g^{[3]}(z)$ is a polynomial with integer coefficients, satisfying the equality
\begin{multline}
\label{eq:pg}
 P_g^{[3]}(z)=z^{-1}P_{g+1}^{[2]}(z)-2\sum_{h=1}^{g} P_h(z)P_{g+1-h}^{[2]}(z)-\\
 -2z^2(1-4z)^2P_{g+1}''(z)-((48g+20)z+5)z(1-4z)P_{g+1}'(z)-\\
 -(48(2g+1)(3g+2)z^2+(60g+44)z+1)P_{g+1}(z).
\end{multline}
In addition, the polynomial~$P_g^{[3]}(z)$ has degree at most $3g+3$ and is divisible by $z^{2g+2}$.
\end{thm}

\begin{proof}
By formula~\eqref{req:eps3} we have
\begin{align}
\label{eq:cg31}
 \mathbf{C}_g^{[3]}(z)&=\sum_{n\ge0}\varepsilon_g(n,3)z^n=\notag\\
& =\sum_{n\ge0}\varepsilon_{g+1}(n+1,2)z^n-
  2\sum_{n\ge0}\sum_{h=0}^{g+1}\sum_{i=0}^{n}\varepsilon_h(i)z^i\cdot\varepsilon_{g+1-h}(n-i,2)z^{n-i}-\notag\\
&\phantom{=\ } -\sum_{n\ge0}(n+1)(2n+1)\varepsilon_{g+1}(n)z^n.
\end{align}
Set the notation
\begin{equation*}
 F_g(z)=\sum_{n\ge0}(n+1)(2n+1)\varepsilon_g(n)z^n.
\end{equation*}
Substituting $F_g(z)$ in the formula~\eqref{eq:cg31}, taking into account that $P_0^{[2]}(z)=z$ and applying the formulas~\eqref{eq:C1}, \eqref{eq:catalan} and \eqref{eq:C2}, we obtain
\begin{align}
\label{eq:cg32}
 \mathbf{C}_g^{[3]}(z)&=z^{-1}\mathbf{C}_{g+1}^{[2]}(z)-2\sum_{h=1}^{g+1}\mathbf{C}_{h}(z)\mathbf{C}_{g+1-h}^{[2]}(z)-2\mathbf{C}_{0}(z)\mathbf{C}_{g+1}^{[2]}(z)-F_{g+1}(z)=\notag\\
& =\biggl(z^{-1}P_{g+1}^{[2]}(z)-2\sum_{h=1}^{g} P_h(z)P_{g+1-h}^{[2]}(z)-2zP_{g+1}(z)\biggr)\frac{1}{(1-4z)^{3g+4,5}}-\notag\\
&\phantom{=\ }\qquad -F_{g+1}(z).
\end{align}

It remains to calculate~$F_{g+1}(z)$. Applying formula~\eqref{eq:C1}, we have

\begin{align}
\label{eq:fg2}
 F_{g+1}(z)&=\sum_{n\ge0}(n+1)(2n+1)\varepsilon_{g+1}(n)z^n=\notag\\
&= 2\sum_{n\ge0}(n+1)(n+2)\varepsilon_{g+1}(n)z^n-3\sum_{n\ge0}(n+1)\varepsilon_{g+1}(n)z^n=\notag\\
&= 2\left(z^2\mathbf{C}_{g+1}(z)\right)''-3\left(z\mathbf{C}_{g+1}(z)\right)'=\notag\\
&= \mathbf{C}_{g+1}(z)+5z\mathbf{C}_{g+1}'(z)+2z^2\mathbf{C}_{g+1}''(z)=\notag\\
&= \frac{2z^2P_{g+1}''(z)}{(1-4z)^{3g+2,5}}+\frac{((48g+20)z+5)zP_{g+1}'(z)}{(1-4z)^{3g+3,5}}+\notag\\
&\phantom{=\ } +\frac{(48(2g+1)(3g+2)z^2+(60g+42)z+1)P_{g+1}(z)}{(1-4z)^{3g+4,5}}.
\end{align}

Substituting formula~\eqref{eq:fg2} in formula~\eqref{eq:cg32}, we obtain the desired equality. 

Let's proof that the expression for  $P_g^{[3]}(z)$ calculated by formula~\eqref{eq:pg} is a polynomial with integer coefficients with~$\deg(P_g^{[3]})\le 3g+3$ and~$z^{2g+2}\mid P_g^{[3]}(z)$. We make use of   properties of polynomials~$P_h(z)$ and~$P_h^{[2]}(z)$ proved in~\cite{APRW1} and~\cite{APRW2}: these polynomials have integer coefficients,~$\deg(P_h)\le 3h-1$, $z^{2h}\mid P_h(z)$, $\deg(P_h^{[2]})\le 3h+1$ and~${z^{2h+1}\mid P_h^{[2]}(z)}$. Hence it follows that all summands in the formula~\eqref{eq:pg} are polynomials with integer coefficients of degree at most~$3g+4$ which are divisible by~$z^{2g+2}$. Therefore, $P_g^{[3]}(z)\in\Z[x]$, 
$\deg(P_g^{[3]})\le 3g+4$ and~$z^{2g+2}\mid P_g^{[3]}(z)$. Let~$a_{3g+2}$ is the coefficient of the polynomial~$P_{g+1}(z)$ at~$z^{3g+2}$. It is easy to see from formula~\eqref{eq:pg} that the coefficient of~$P_g^{[3]}(z)$ at~$z^{3g+4}$ is equal to
$$(-32(3g+2)(3g+1)+4(48g+20)(3g+2)-48(2g+1)(3g+2))a_{3g+2}=0,$$ whence it follows that $\deg(P_g^{[3]})\le 3g+3$.
\end{proof}

\begin{cor}
\label{cor:p3}
For $g=0,1,\ldots,4$ the polynomials~$P_g^{[3]}(z)$ are the following:
\begin{align*}
P_0^{[3]}(z)&=2z^2\, (4z+3),\\
P_1^{[3]}(z)&=12z^4\, \left(68\,z^2 + 207\,z + 45\right),\\
P_2^{[3]}(z)&=6z^6\, \left(27592\,z^3 + 197646\,z^2 + 137934\,z + 15015 \right),\\
P_3^{[3]}(z)&=8z^8\, \left(7468348\,z^4 + 98362965\,z^3 + \right.\\
&\qquad\qquad\qquad\left.+ 143262162\,z^2 + 46335375\,z + 3132675 \right),\\
P_4^{[3]}(z)&=90z^{10}\left(383244280\,z^5 + 8028110250\,z^4 + 20036503284\,z^3 + \right.\\
&\qquad\qquad\qquad\left.+12962876908\,z^2 + 2494416504\,z + 117515475\right).
\end{align*}
\end{cor}

\begin{proof}
Substituting the formulas for  $P_i(z)$ and $P_i^{[2]}(z)$ from remark~\ref{rem:p1p2} in the formula~\eqref{eq:pg}, we obtain the desired formulas.
\end{proof}

\begin{cor}[{see.~\cite[theorem 6]{PR}}]
\label{cor:eps03}
The numbers $\varepsilon_0(n,3)$ for $n\geq 0$ satisfy the following equation:
\begin{equation*}
\varepsilon_0(n,3)=\frac{(8n+5)(n-1)n(n+1)}{210}\,C_{2n+1}^{n}=\frac{8n+5}{35}\,C_{2n+1}^{n}C_{n+1}^{3}.
\end{equation*}
\end{cor}

\begin{proof}
Substituting the expression for $P_0^{[3]}(z)$ from  corollary~\ref{cor:p3} in formula~\eqref{C3}, we obtain the following:
\begin{equation}
\label{eq:c03}
 \mathbf{C}_0^{[3]}(z)=\frac{8z^3+6z^2}{(1-4z)^{4,5}}.
\end{equation}
Note, that
\begin{align}
\label{eq:-4,5}
 (1-4z)^{-4,5}&=\sum_{n\ge0}\binom{-4,5}{n}(-4z)^n=\sum_{n\ge0}\frac{4^n\cdot\frac{9}{2}\cdot\frac{11}{2}\cdot\ldots\cdot\frac{2n+7}{2}}{n!}\,z^n=\notag\\
 &= \sum_{n\ge0}\frac{2^n(2n+7)!!}{7!!\,n!}\,z^n.
\end{align}
Denote by $c_n$ the coefficient of this series at~$z^n$  (for~$n<0$ we set~$c_n=0$). Then by formulas~\eqref{eq:c03} and~\eqref{eq:-4,5} we obtain that
\begin{align*}
 \varepsilon_0(n,3)&=8c_{n-3}+6c_{n-2}=\\
 &= 8\cdot\dfrac{2^{n-3}(2n+1)!!(n-2)(n-1)n}{105n!}+6\cdot\dfrac{2^{n-2}(2n+3)!!(n-1)n}{105n!}=\\
 &= \frac{2^{n-1}(2n+1)!!(n-1)n(8n+5)}{105n!}=\frac{(8n+5)(n-1)n(n+1)}{210}\,C_{2n+1}^{n}.
\end{align*}
\end{proof}

\begin{cor}
\label{cor:eps13}
The numbers~$\varepsilon_1(n,3)$ for $n\geq 0$ satisfy the following equation:
\begin{equation*}
\varepsilon_1(n,3)=\frac{808n^2+99n-454}{3003}\,C_{2n+1}^{n}C_{n+1}^{5}.
\end{equation*}
\end{cor}

\begin{proof}
Substituting the expression for $P_1^{[3]}(z)$ from  corollary~\ref{cor:p3} in formula~\eqref{C3}, we obtain the following:
\begin{equation*}
 \mathbf{C}_0^{[3]}(z)=\frac{816z^6+2484z^5+540z^4}{(1-4z)^{7,5}}.
\end{equation*}
Similarly to the proof of corollary~\ref{cor:eps03} we obtain that
\begin{equation*}
 (1-4z)^{-7,5}=\sum_{n\ge0}s_nz^n,
\end{equation*}
where $s_n=\frac{2^n(2n+13)!!}{13!!\,n!}$ and $s_n=0$ for $n<0$. Then
\begin{align*}
 \varepsilon_1(n,3)&=816s_{n-6}+2484s_{n-5}+540s_{n-4}=\\
 &= 816\cdot\frac{2^{n-6}(2n+1)!!(n-5)(n-4)(n-3)(n-2)(n-1)n}{13!!\,n!}+\\
 &\phantom{=\ } +2484\cdot\frac{2^{n-5}(2n+3)!!(n-4)(n-3)(n-2)(n-1)n}{13!!\,n!}+\\
 &\phantom{=\ } +540\cdot\frac{2^{n-4}(2n+5)!!(n-3)(n-2)(n-1)n}{13!!\,n!}=\\
 &= \frac{3\cdot2^{n-3}(2n+1)!!(n-3)(n-2)(n-1)n(808n^2+99n-454)}{13!!\,n!}=\\
 &= \frac{808n^2+99n-454}{3003}\,C_{2n+1}^{n}C_{n+1}^{5}.
\end{align*}
\end{proof}

\smallskip
Translated by D.\,V.\,Karpov.

\end{document}